\documentclass[10pt,leqno]{amsart}
\usepackage{indentfirst,csquotes}

\usepackage{paralist,fancyhdr,etoolbox}

\usepackage[ngerman,english]{babel} 
\usepackage[utf8]{inputenc} % Required for inputting international characters
\usepackage[T1]{fontenc} % Output font encoding for international characters
\usepackage{xcolor,colortbl}

\usepackage{amsmath}
\usepackage{amssymb}
\usepackage{amsthm}
\usepackage{amscd}
\usepackage{amsfonts}
\usepackage{MnSymbol}
\usepackage{mathrsfs}
\usepackage{mathtools}
\usepackage{bbm}
\usepackage{algorithm2e}
\usepackage{listings}
\usepackage{quiver}
\usepackage{adjustbox}
\usepackage{tabularx}
\usepackage{enumitem}

\usepackage{pgf,tikz}
\usepackage{tikz-cd}
\usetikzlibrary{matrix,arrows,calc}
\usepackage{pgfplots}
\usepackage{graphicx}
\usepackage{caption}
\usepackage{subcaption}
\usepackage{framed}
\usepackage{float}
\usepackage{media9}

\usepackage{hyperref}
\usepackage[onehalfspacing]{setspace} %Zeilenabstand 1,5
\usetikzlibrary{decorations.markings}
\tikzset{
	halfarrow/.style={postaction={decorate},
		decoration={markings,mark=at position .5 with
			{\arrow{stealth}}}}}
\newcommand{\R}{\mathbb{R}}

\newcommand{\N}{{\mathbb N}}

\newcommand{\Z}{\mathbb{Z}}

\newcommand{\OP}{\mathcal{O}}
\newcommand{\CP}{\mathcal{C}}

\newtheorem{theorem}{Theorem}[section]

\newtheorem{prop}[theorem]{Proposition}

\newtheorem{lemm}[theorem]{Lemma}

\newtheorem{lemmdefi}[theorem]{Lemma(and Definition)}
\theoremstyle{definition}
\newtheorem{propdefi}[theorem]{Proposition(and Definition)}

\newtheorem{defi}[theorem]{Definition}
\newtheorem{bsp}[theorem]{Example}
\newtheorem{bem}[theorem]{Remark}

\AtBeginDocument{}\AtBeginDocument{}
\AtBeginDocument{}
\newsavebox{\imagebox}

\DeclareMathOperator{\aff}{aff}

    \usepackage[backend=bibtex,style=alphabetic, maxnames=99]{biblatex}
\hypersetup{ colorlinks=false, linkcolor=black, filecolor=black, urlcolor=black }

\begin{document}
\title{2-levelness of Marked Poset Polytopes\\
and the Ehrhart Polynomial} %%%%%%%%%%%%
\author[J. Stricker]{Jan Stricker}
%\date{\today}
\address{Institute for Mathemtics, Goethe University Frankfurt, Robert-Mayer-Str. 10,\newline
60325 Frankfurt am Main, Germany}
\email{stricker@math.uni-frankfurt.de}
\maketitle

\let\thefootnote\relax
\footnotetext{MSC: Primary: 52B20 $\cdot$ Secondary: 52B12 $\cdot$ 06A07 $\cdot$ 51M20 $\cdot$ 05A15.\\
\textit{Key words and phrases:} 2-level polytopes, marked posets, marked order polytopes, marked chain polytopes, marked chain-order polytopes, Ehrhart polynomial
} %%%%%%%%%%

\begin{abstract}
It is already known that order polytopes and chain polytopes are always 2-level polytopes. In general, this is not true for marked order and marked chain polytopes.\\
We study the geometry of marked order polytopes, marked chain polytopes, and marked chain-order polytopes, providing a comprehensive characterisation of  2-levelness for these polytopes. Furthermore, we present an exact formula for the Ehrhart polynomial of marked order polytopes. Because of their connection to marked chain and marked chain-order polytopes, this polynomial is also the Ehrhart polynomial of these polytopes.
\end{abstract} %%%%%%%%%

\section{Introduction}
	We call a polytope $Q\subseteq \R^d$ \textit{2-level}, if for every facet of $Q$ there exits a parallel hyperplane, such that all vertices of $Q$ are in the facet or in the parallel hyperplane. Alternatively, a polytope $Q$ is considered to be 2-level if and only if it possesses theta-rank 1 \cite{gouveia2010theta}, or it exhibits a slack matrix with entries restricted to $\{0, 1\}$ \cite{bohn2019enumeration}. These definitions, originating from the realms of semidefinite programming, statistics and polyhedral combinatorics respectively, underscore the inherent multidisciplinary nature of 2-level polytopes within mathematics. A 2-level polytope can be viewed as a generalization of well-known polytopes, including Birkhoff polytopes \cite{ziegler2006}, Hanner polytopes \cite{hanner1956intersections}, and Hansen polytopes \cite{hansen1977certain}. The class of 2-level polytopes also encompasses spanning tree polytopes of series-parallel graphs \cite{grande2017theta}, stable matching polytopes \cite{gusfield1989stable}, certain min up/down polytopes \cite{lee2004min}, and stable set polytopes of perfect graphs \cite{chvatal1975certain}. There is also a stronger connection between 2-level polytopes and compressed polytopes. The main difference between these classes, is that compressed polytopes are always integral 2-level polytopes, where each 2-level polytope is an affine image of a compressed polytope, where compressed polytopes are characterised as integral polytopes all of whose pulling triangulations are unimodular \cite{compressed}. This broad spectrum of connections illustrates the pervasive presence of 2-level polytopes across diverse mathematical domains.\\
To any finite poset $(P,\prec)$ Stanley associated two polytopes, which are lattice polytopes with the same Ehrhart polynomial \cite{stanley1986two}. These are the \textit{order polytope} and the \textit{chain polytope}. Moreover, these poset polytopes were already topic of a lot of research and it was proven, that order polytopes and chain polytopes are 2-level polytopes \cite{ohsugi2001compressed}.
Inspired by the representation theory of complex semi-simple Lie algebras, specifically within the context of PBW-degenerations, Ardila, Bliem, and Salazar \cite{ardila2011gelfand} introduced the concept of \textit{marked order polytopes} and \textit{marked chain polytopes}, generalizing Stanley's order and chain polytope. These polytopes are associated to \textit{marked posets}. Given a finite poset $P$, we define a set of marked elements $P^*\subseteq P$, which includes the extremal elements of $P$. Then we can define a marking $\lambda:P^*\rightarrow\Z$, which gives all elements of $P^*$ an integral value preserving the ordering of $P$. These values introduces bounds for the elements of marked order polytopes and marked chain polytopes.
Their investigation revealed that these polytopes are also indeed lattice polytopes, and for a given marked poset, they have identical Ehrhart polynomials.\\
Motivated by the recent work on linear degeneration of flag varieties \cite{cerulli2017linear}, Fang and Fourie introduced \textit{marked chain-order polytopes} \cite{fang2016marked}. These polytopes represent a class of polytopes, which combines features of marked order and marked chain polytopes. For each order ideal of the poset, one imposes chain conditions on the coordinates in the order ideal, and order conditions on the coordinates in its complement. Significantly, Fang and Fourie established that these marked chain-order polytopes form a family of lattice polytopes with Ehrhart equivalence. In addition, marked chain-order polytopes generalize the classes of marked order and marked chain polytopes. Fang, Fourier and Pegel were also able to characterise reflexivity for this class of polytopes \cite{fang2018minkowski}.\\

In this paper, we mainly focus on characterising 2-levelness for marked order polytopes, marked chain polytopes and marked chain-order polytopes. The main result will be, that marked order and marked chain polytopes are 2-level, if and only if they are already an affine image of an order or a chain polytope.
At first we have to make some general definitions and formulated some important theorems in Section \ref{section2}. These definitions and theorems will be useful throughout the whole paper.
In Theorem \ref{markorder2l}, Theorem \ref{markchain2l} and Theorem \ref{chainorder2l} we will formulate a complete characterisation of 2-levelness for marked order, marked chain and marked chain-order polytopes respectively. These will use different techniques, depending, on what we know about these classes of polytopes. Yet, the characterisations will be very similar to each other.\\
Furthermore, we show a concrete formula for the Ehrhart polynomial for marked order polytopes in Theorem \ref{ehrmarkorder}. Also, this will be a concrete formula for all marked chain-order polytopes. The complex nature of these polytopes will result in a very complex formula, but we will present an example of a marked poset, which will have a nice Ehrhart polynomial.\\

\textbf{Acknowledgements.} This paper was written during a student exchange supervised by Professor Akihiro Higashitani at Osaka University. I am grateful to Professor Akihiro Higashitani and my tutor, Max Kölbl, for their guidance and support during my research at Osaka University.

	\section{Review of relevant Theorems and Definitions}\label{section2}
The set $\R$ (resp. $\Z$, $\N$) of real (resp. integral, natural) numbers has the usual total order. For us the set of natural numbers $\N$ will include the 0. The set $[m]$ denotes the set $\{1,2,\ldots,m\}$.
For a convex polytope $Q \subseteq \R^d$, we denote the set of vertices as $V(Q)$. The set $\aff(F)$ of a set $F\subseteq \R^d$ is the affine hull of $F$. For a poset $(P,\prec)$ we will often just use the notation $P$. We still then use $\prec$ for describing the order of $P$. We denote the set of minimal and maximal elements of $P$ as $\min(P)$ and $\max(P)$.\\
	
	Let us start by defining the polytopes we are working with and some of their basic properties.
	
		\begin{defi}\label{Def2level}
		Let $Q \subseteq \R^d$ be a polytope. We call $Q$ \textit{2-level}, if for any facet $F\subseteq Q$, there exists $t \in \R^d$, such that all vertices $v \in V(Q)\backslash V(F)$ of $Q$ lies in $t+\aff(F)$.
		% $ v \in \aff(F)+t$
		%\begin{align*}
		%\exists t \in \R^d \text{ sodass } \forall v \in V(P)\slash V(F) \text{ gilt } v \in \aff(F+t) 
		%\end{align*}
	\end{defi}

	\begin{prop}\label{aquivalent}
	Let $Q \subseteq \R^d$ be a polytope. Then $Q$ is a 2-level polytope, if one of the following is satisfied:\\
	(i) For each hyperplanes $H$, which define a facet $F$,
	\begin{align*}
	\text{there exists } t \in \R^d \text{, such that } V(Q) \subseteq H \cup (H+t).
\end{align*}	

	(ii) Let $\{x \in \R^d\,:\,a_i^tx\leq b_i\, ,\,i\in [m]\}$ be a finite irredundant description for the polytope $Q$ with $a_i \in \R^d$ and $b_i\in \R$. For all $i \in [m]$ there exists $m_i \in \R$, such that every vertex of $Q$ has to fulfill either $a_i^tx=b_i$ or $a_i^tx=b_i+m_i$.
	\end{prop}
	
	\begin{defi}
	We call a map $f:\R^d\rightarrow \R^d$ an \textit{affine transformation}, if $f$ is of the form $x \mapsto Ax+x_0$, where $A$ is a nonsingular square matrix and $x_0 \in \R^d$.\\
	For a set $F\subseteq\R^d$ we call $f(F)$ the \textit{affine image} of $F$.\\
	Let $Q,Q'\subseteq \R^d$ be polytopes. We call $Q$ and $Q'$ \textit{affinely isomorphic}, if there exists an affine transformation $f$, such that $f(Q)=Q'$.
	\end{defi}
	
	%It is important to mention, that the class of 2-level polytopes is closed under affine transformations.
	
	\begin{lemm}
	Let $Q\subseteq \R^d$ be a 2-level polytope. Then for every affine transformation $f$, the polytope $f(Q)$ is still a 2-level polytope.
	\end{lemm}
	\begin{proof}
	This is trivial, since parallel objects stay parallel after affine transformations.
	\end{proof}
	
Before defining the notations of marked posets and marked poset polytopes, we will remind ourselves of the definitions of order polytopes and chain polytopes.

\begin{defi}\textnormal{\cite[Def 1.1]{stanley1986two}}
The \textit{order} polytope $\OP(P)$ of the poset $P$ is the subset of $\R^P$ defined by the conditions
\begin{align*}
	(1)&~0 \leq x_p \leq 1, \text{ for all }p\in P \text{ and }\\
	(2)&~x_p \leq x_q, \text{ for the } p,q\in P \text{ with }p\preceq q.
\end{align*}
\end{defi}

\begin{defi}\textnormal{\cite[Def 2.1]{stanley1986two}}
The \textit{chain} polytope $\CP(P)$ of the poset $P$ is the subset of $\R^P$ defined by the conditions
\begin{align*}
	(1)&~0 \leq x_c, \text{ for all }c\in P \text{ and }\\
	(2)&~x_{c_1} + \cdots + x_{c_k} \leq 1, \text{ for all (maximal) chains } c_1 \prec \cdots \prec c_k \text{ in } P. 
\end{align*}
\end{defi}
	
We will define next marked posets. This definition let us directly define marked order, marked chain and marked chain-order polytopes.	
	
	\begin{defi}\textnormal{\cite[Def 2.1]{pegel2018face}}
A \textit{marked poset} $(P, \lambda)$ is a finite poset $P$ together with an induced subposet $P^*\subseteq P$ of \textit{marked elements} and an order-preserving \textit{marking} $\lambda : P^* \rightarrow \R$.\\
The marking $\lambda$ is called \textit{strict} if $\lambda(a) < \lambda(b)$ whenever $a \prec b$. A map $f : (P, \lambda) \rightarrow (P', \lambda')$ between marked posets is an order-preserving map $f : P \rightarrow P'$ such that
$f(P^*) \subseteq (P')^*$ and $\lambda'(f(a)) = \lambda(a)$ for all $a \in P^*$.
\end{defi}

\begin{defi}\textnormal{\cite[Def 3.19]{pegel2018face}}
A marked poset $(P, \lambda)$ is called \textit{regular} if for each
covering relation $p \prec q$ in $P$ and $a, b \in P^*$ such that $a \preceq q$ and $p \preceq b$, we have $a = b$ or $\lambda(a) < \lambda(b)$.
\end{defi}

From this property, we can derive useful insights for the marked poset.

\begin{prop}\textnormal{\cite[Prop 3.20]{pegel2018face}}
	Let $(P, \lambda)$ be a regular marked poset. For $(P, \lambda)$ all following statements are true.
	\begin{enumerate}[label=(\arabic*)]
		\item The marking $\lambda$ is strict.
		\item There exists no cover relations between marked elements.
		\item Every element in $P$ covers at most one marked element and is at most covered by one marked element.
	\end{enumerate}
\end{prop}

	\begin{defi}\textnormal{\cite[Def 3.3]{ardila2011gelfand}}
	Let $(P, \lambda)$ be a marked poset with $\min(P)\cup\max(P)\subseteq P^*$. The \textit{marked order polytope} $\OP(P, \lambda)$ associated to $(P, \lambda)$ is the set of all $x \in \R^{P\backslash P^*}$ satisfying the following conditions:
	\begin{align*}
		(1)&~\lambda(a) \leq x_p \leq \lambda(b), \text{ for all }p\in P\backslash P^* \text{ and } a,b \in P^* \text{ with }a \prec p \prec b \text{ and }\\
		(2)&~x_p \leq x_q, \text{ for the } p,q\in P\backslash P^* \text{ with }p\preceq q.
	\end{align*}
\end{defi}

\begin{defi}\textnormal{\cite[Def 3.3]{ardila2011gelfand}}\label{defmarkchain}
	Let $(P, \lambda)$ be a marked poset with $\min(P)\cup\max(P)\subseteq P^*$. The \textit{marked chain polytope} $\CP(P, \lambda)$ associated to $(P, \lambda)$ is the set of all $x \in \R_{\geq 0}^{P\backslash P^*}$ satisfying the following condition:
	\begin{align*}
		(1)&~x_{p_1} + x_{p_2} + \ldots + x_{p_k} \leq \lambda(b)-\lambda(a), \text{ for each (maximal) chain } \\
		&~a\prec p_1\prec p_2\prec\ldots\prec p_k\prec b \text{ in } P \text{ with } a,b \in P^*.
	\end{align*}
\end{defi}

\newpage

\begin{defi}\textnormal{\cite[Def 1.2]{fang2018minkowski}}
Let $P = P^* \cup C \cup O$ be a partition of a poset $P$ with $\min(P)\cup\max(P)\subseteq P^*$ and $\lambda$ a marking. The elements of $C$ and $O$ are called \textit{chain elements} and \textit{order elements}, respectively. The \textit{marked chain-order polytope} $\OP_{C,O}(P, \lambda) \subseteq \R^P$ is the
set of all $x = (x_p)_{p\in P} \in R^P$ satisfying the following conditions:\begin{align*}
(1)&\text{ for any } a \in P^*,~ x_a = \lambda(a);\\
(2)&\text{ for } p \in C, x_p \geq 0;\\
(3)&\text{ for each saturated chain } a \prec p_1 \prec \cdots \prec p_r \prec b \text{ with } a, b \in P^* \cup O, p_i \in C,r \geq 0,\\
&\text{ we have } x_{p_1} + \cdots+ x_{p_r} \leq x_b - x_a.
\end{align*}
\end{defi}

From here on we will always use $P^*$ for the set of marked elements for a poset $P$. We also assume $\min(P)\cup\max(P)\subseteq P^*$ for each marked poset.

\begin{bem}
When a partition $P = P^* \cup C \cup O$ is given, we write the points of $\R^P$ as $x =(\lambda, x_C , x_O)$ with $\lambda \in R^{P^*}$, $x_C \in R^C$ and $x_O \in R^O$. Since the coordinates in $P^*$ are fixed for the points of $\OP_{C,O}(P, \lambda)$, we usually consider the projection of $\OP_{C,O}(P, \lambda)$ in $\R^{P\backslash P^*}$ instead, keeping the same notation to write $(x_C , x_O) \in \OP_{C,O}(P, \lambda)$ instead of $(\lambda, x_C , x_O) \in \OP_{C,O}(P,\lambda)$.
\end{bem}

Especially in Section \ref{section3}, where we will be looking at marked order polytopes, we need to understand the faces of these polytopes. Thankfully Pegel already proved a lot of useful theorems \cite{pegel2018face}.

	\begin{defi}\textnormal{\cite[Def 3.1]{pegel2018face}}
Let $Q = \OP(P, \lambda)$ be a marked order polytope. To each $x \in Q$ we
associate a partition $\pi_x$ of $P$ induced by the transitive closure of the relation
\begin{align*}
p \sim_x q \text{ if } x_p = x_q
\end{align*}
and $p, q$ are comparable.
\end{defi}

\begin{defi}
Given any partition $\pi$ of $P$, we call a block $B \in \pi$ \textit{free} if $P^* \cap B = \emptyset$ and denote
by $\tilde{\pi}$ the set of all free blocks of $\pi$.
\end{defi}

	\begin{prop}\textnormal{\cite[Prop 3.2]{pegel2018face}}\label{orderpart}
Let $x \in Q = \OP(P, \lambda)$ be a point of a marked order polytope with associated partition $\pi = \pi_x$. We define the set
$$F_x := \{ y \in Q : y \text{ is constant on the blocks of } \pi\}.$$
Then it holds that $\dim(\aff(F_x)) = |\tilde{\pi}|$.
\end{prop}

%\begin{coro}\textnormal{\cite[Cor 3.6]{pegel2018face}}
%Let $Q = \OP(P, \lambda)$ be a marked order polytope. The poset $F(Q)\backslash \{\emptyset\}$
%of non-empty faces of Q is isomorphic to the induced subposet of the partition lattice on $P$
%given by all face partitions of $(P, \lambda)$.
%\end{coro}	

\begin{defi}\textnormal{\cite[Def 3.7]{pegel2018face}}
Let $(P, \lambda)$ be a marked poset. A partition $\pi$ of $P$ is \textit{connected} if the
blocks of $\pi$ are connected as induced subposets of $P$. It is $P$\textit{-compatible}, if the relation
$\preceq$ defined on $\pi$ as the transitive closure of
\begin{align*}
B \preceq C \text{ if } p \preceq q \text{ for some } p \in B, q \in C
\end{align*}
is anti-symmetric. In this case $\preceq$ is a partial order on $\pi$. A $P$-compatible partition $\pi$
is called $(P, \lambda)$\textit{-compatible}, if whenever $a \in B \cap P^*$ and $b \in C \cap P^*$ for some blocks
$B \preceq C$, we have $\lambda(a) \leq \lambda(b)$.
\end{defi}

\begin{theorem}\textnormal{\cite[Thm 3.14]{pegel2018face}}\label{orderface}
A partition $\pi$ of a marked poset $(P, \lambda)$ is a face partition if and only if it
is $(P, \lambda)$-compatible, connected and the induced marking on $(P\slash \pi, \lambda\slash \pi)$ is strict.
\end{theorem}

\begin{prop}\textnormal{\cite[Prop 4.2]{pegel2018face}}\label{productposet}
Let $(P_1, \lambda_1)$ and $(P_2, \lambda_2)$ be marked posets on disjoint sets. Let the marking $\lambda_1 \sqcup \lambda_2 : P^*_1 \sqcup P^*_2 \rightarrow \R$ on $P_1 \sqcup P_2$ be given by $\lambda_1$ on $P^*_1$ and $\lambda_2$ on $P^*_2$. The marked order polytope $\OP(P_1 \sqcup P_2, \lambda_1 \sqcup \lambda_2)$ is equal to the product $\OP(P_1, \lambda_1) \times \OP(P_2, \lambda_2)$ under the canonical identification $\R^{P_1\sqcup P_2} = \R^{P_1} \times \R^{P_2}$.
\end{prop}

	\section{2-level Marked Order Polytopes}\label{section3}
	
Before characterising the 2-level marked order polytopes, it is useful to make some observations.\\
Firstly, we only want to work with connected posets, therefore we prove that we can separate these components, while working with 2-levelness.	
	
	\begin{lemm}\label{conmaxmin}
	Let $(P,\lambda)$ be a marked poset, for which the Hasse-diagram consists of connected, disjoint components, where each component only has one maximal and one minimal marked element. Then $\OP(P, \lambda)$ is affinely isomorphic to an order polytope.
	\end{lemm}
	\begin{proof}
	Let $P_1,\ldots,P_k$ be the connected, disjoint components of $(P,\lambda)$ with $\lambda_i=\lambda\mid_{P_i}$. Since each $P^*_i$ only contains one minimal and one maximal element, they are affinely isomorphic to an order polytope, by changing the marking of these elements to $0$ and $1$. This change is an affine transformation without changing the structure of the polytope.\\
	From Proposition \ref{productposet} we know, that 
$$
\OP(P_1 \sqcup \cdots \sqcup P_k, \lambda_1 \sqcup \cdots \sqcup \lambda_k) = \OP(P_1,\lambda_1) \times \cdots \times \OP(P_k,\lambda_k).$$
Since the affine transformations only affect each marked order polytope separately, we can combine them to affinely transform  $\OP(P_1 \sqcup \cdots \sqcup P_k, \lambda_1 \sqcup \cdots \sqcup \lambda_k)$ such that all markings are only $0$ and $1$. We call this new marked order polytope $\mathcal{P}$ with poset $P'$ and marking $\lambda'$. Since all the components of $P'$ are disjoint and have a maximal marking of $1$ and a minimal marking of $0$, we can connect them, by deleting the individual marked elements and connect all elements of $P'$ to one maximal element $\hat{1}$ and one minimal element $\hat{0}$. This does not change the marked order polytope $\mathcal{P}$. From \cite[Section 3.2]{ardila2011gelfand} we know that $\mathcal{P}$ is now an order polytope and by construction affine isomorphic to $\OP(P,\lambda)$.
	\end{proof}
	
We use the theorems, which describe faces of marked order polytopes to define the exact facets of marked order polytopes.
	
		\begin{lemm}\label{markedfacet}
	Let $(P,\lambda)$ be a regular marked poset. Then the facet defining equalities of $\OP(P, \lambda)$ are of the form
	\begin{align*}
	&x_p = x_q \text{ for } p\prec q \text{ in } P &\text{ and} \\
	&x_p = \lambda(a) \text{ for } p \prec a \text{ or } a\prec p \text{ with } p \in P \text{ and } a \in P^*.&
	\end{align*}
	\end{lemm}
	\begin{proof}
	Let $(P,\lambda)$ be a regular marked poset and $|P\backslash P^*|=d$. We know from Proposition \ref{orderpart} and Theorem \ref{orderface}, that facets of $\OP(P, \lambda)$ are defined by partitions of $P$ with $d-1$ free blocks.\\
	Since $|P\backslash P^*|=d$ there are only two options for how we can achieve $d-1$ free blocks. The first one is that $x_p = x_q \text{ for a a covering } p\prec q \text{ in } P$, because then $q$ and $p$ are in the same block and exactly two elements of $P\backslash P^*$ can only be in the same block, by being directly connect by a covering. The second option is that $x_p = \lambda(a) \text{ for some } p \in P \text{ and some } a \in P^*$ for obvious reasons.
	\end{proof}
	
	With these two lemmas we can now characterise marked order polytopes, which are 2-level polytopes.
	
	\begin{theorem}\label{markorder2l}
	Let $(P,\lambda)$ be a regular marked poset.
	The following conditions are equivalent:
	\begin{enumerate}[label=(\alph*)]
	\item $\OP(P,\lambda)$ is a 2-level polytope.
	\item Each connected component of the poset $P$ has one unique maximal and one unique minimal marked element.
	\item $\OP(P,\lambda)$ is affinely isomorphic to an order polytope.
	\end{enumerate}
	\end{theorem}
	\begin{proof}
	Let $(P,\lambda)$ be a regular marked poset and $d=|P\backslash P^*|$.\\
	Because of Lemma \ref{conmaxmin} we see \textit{(b)} $\Rightarrow$\textit{(c)}. Since every compressed polytope is a 2-level polytope and order polytopes are compressed, \textit{(c)} $\Rightarrow$\textit{(a)} is trivial \cite{ohsugi2001compressed}.\\
	Showing \textit{(a)} $\Rightarrow$\textit{(b)} is more complicated. Assume $\OP(P,\lambda)$ is a 2-level polytope. From Lemma \ref{markedfacet} we know, that there are only two kinds of facets. We want to get some informations about the vertices from $P$ by examining these two cases.\\
	\textit{Case 1:} Let $F$ be a facet defined by $x_p=\lambda(a)$ for some $p \in P$ and $a \in P^*$. W.l.o.g. $p$ covers $a$ in $P$.\\
	Because $\OP(P,\lambda)$ is a 2-level polytope, every vertex $v$ has to fulfill one of the following equations:
	$$
	v_p=\lambda(a) \text{ or } v_p=c \text{ for } c \in \R.
	$$
	By Proposition \ref{orderpart}, the partition for the vertex has to have zero free blocks. Therefore $c=\lambda(b)$ for some $b\in P^*$. We can conclude, that $a$ is the highest lower bound for $p$ and $b$ is the lowest higher bound for $p$.\\
	As a result, every time an element $p$ of $P\backslash P^*$ is less than an element $b$ of $P^*$ or is greater than an element $a$ of $P^*$, every vertex has $x_p=\lambda(a)$ or $x_p=\lambda(b)$.\\
	\textit{Case 2:} Let $F$ be a facet defined by $x_p=x_q$ for some $p,q \in P\backslash P^*$ and $p\prec q$.\\
	We want to show, that $p$ and $q$ have the same value for their highest lower bound marking and lowest higher bound marking. Assume $a_p$ and $a_q$ are their respective lower bounds and $b_p$ and $b_q$ their respective higher bounds.\\
	We want to see now, that $x_p$ and $x_q$ cannot have the same value $k \in \R$ in the whole facet. Assume they are equal to a value $k \in \R$ in the whole facet. Since there are vertices in the facet and all vertices have markings as coordinates, $k = \lambda(e)$ for some $e \in P^*$. Therefore the equation $x_p=x_q=\lambda(e)$ holds for the whole facet. Hence $p$, $q$ and $e$ are in the same block in the partition of the facet. This is a contradiction, since this partition has now at most $d-2$ free blocks and cannot describe a facet any more.\\
	From this we can conclude, that $\lambda(a_p)<\lambda(b_p)$ and $\lambda(a_q)<\lambda(b_q)$. We also know $\lambda(a_p)\leq \lambda(a_q) < \lambda(b_p)\leq \lambda(b_q)$, since $x_p=x_q$ for some vertices.\\
	Assume $\lambda(b_p)<\lambda(b_q)$. We get that there are non-crossing connections from $p$ and $q$ to $b_p$ and $b_q$ respectively. And $b_p$ has to be connected to a lower element then $q$ because of the ordering of $P$. In addition, we know from the 2-levelness of the polytope, all vertices have to fulfill either $x_p=x_q$ or $x_p=x_q+c$ for some $c \in \R$.\\
	We can conclude from Theorem \ref{orderface} that there exist vertices $v$ and $w$ with $v_p=\lambda(b_p)$ and $v_q=\lambda(b_q)$, $w_p=\lambda(a_p)$ and $w_q=\lambda(b_q)$. From $x_p=x_q+c$ we have to conclude, that $\lambda(a_p)=\lambda(b_p)$ which is absurd. Therefore $\lambda(b_p)=\lambda(b_q)$ and we can combine $b_p$ and $b_q$ in the poset since $p\prec q$ (except when $b_p=b_q$).
	Analogously we can conclude $\lambda(a_p)=\lambda(a_q)$ and combine $a_p$ and $a_q$ in the poset (except when $a_p=a_q$).\\
	As a result from Theorem \ref{orderface}, and the assumptions on $P$, we know, that every covering of $P$ gives us a facet. In conclusion with the two cases, we can see, that all elements of a component of the poset have exactly one maximal and one minimal element.
	\end{proof}

\section{2-level Marked Chain Polytopes}\label{section4}

We want to have a similar characterisation for marked chain polytopes. The difficulty in working with marked chain order polytopes lies in not knowing as much about the faces of marked chain polytopes. Therefore, we need some tricks.

\begin{lemm}\label{chainiso}
Let $(P,\lambda)$ be a regular marked poset. The $\CP(P,\lambda)$ is a chain polytope if all describing inequalities of the polytope are of the form
\begin{align*}
x_i \geq 0 ~~&\text{for }  i\in P\backslash P^* &\text{ and }\\ 
\sum_{i\in I} x_i \leq 1 ~~& \text{for } I \subseteq P\backslash P^*.&
\end{align*}
\end{lemm}
\begin{proof}
If $\CP(P,\lambda)$ is only described by inequalities
\begin{align*}
x_i \geq 0 \text{ for } i\in P\backslash P^* \text{ and }\sum_{i\in I} x_i \leq 1 \text{ for } I \subseteq P\backslash P^*,
\end{align*}
we can construct a poset for the elements $P\backslash P^*$, where all maximal chains are the elements in the $I$ from the describing inequalities.\\
For this poset the chain polytope is obviously the polytope $\CP(P,\lambda)$.
\end{proof}

\begin{figure}[h!]
%\begin{tikzcd}[ampersand replacement=\&, column sep = 0.1]
% https://q.uiver.app/#q=WzAsNixbMSwxLCJ4XzIiXSxbMSwyLCJ4XzEiXSxbMiwxLCIyIl0sWzAsMiwiMiJdLFsxLDMsIjEiXSxbMSwwLCIzIl0sWzUsMCwiIiwwLHsic3R5bGUiOnsiaGVhZCI6eyJuYW1lIjoibm9uZSJ9fX1dLFswLDMsIiIsMCx7InN0eWxlIjp7ImhlYWQiOnsibmFtZSI6Im5vbmUifX19XSxbMCwxLCIiLDAseyJzdHlsZSI6eyJoZWFkIjp7Im5hbWUiOiJub25lIn19fV0sWzEsMiwiIiwwLHsic3R5bGUiOnsiaGVhZCI6eyJuYW1lIjoibm9uZSJ9fX1dLFsxLDQsIiIsMCx7InN0eWxlIjp7ImhlYWQiOnsibmFtZSI6Im5vbmUifX19XV0=
\begin{subfigure}{0.49\linewidth}
\centering
\begin{tikzcd}[ampersand replacement=\&]
	\& 3 \\
	\& {x_2} \& 2 \\
	2 \& {x_1} \\
	\& 1
	\arrow[no head, from=1-2, to=2-2]
	\arrow[no head, from=2-2, to=3-1]
	\arrow[no head, from=2-2, to=3-2]
	\arrow[no head, from=3-2, to=2-3]
	\arrow[no head, from=3-2, to=4-2]
\end{tikzcd}
\caption{a marked poset $(P,\lambda)$}
\label{fig:bspmchainpos1}
\end{subfigure}
\begin{subfigure}{0.49\linewidth}
\centering
\includegraphics[width=0.7\linewidth]{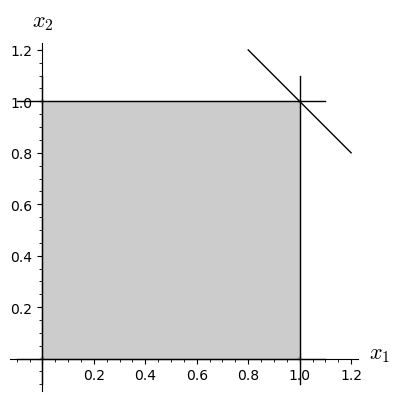}
\caption{the marked chain polytope $\CP(P,\lambda)$}
\label{fig:bspmchainpos2}
\end{subfigure}
\caption{The marked poset $(P,\lambda)$ from Example \ref{bsplem3-1} and the corresponding marked chain polytope $\CP(P,\lambda)$.}
\label{fig:bspmchainpos}
\end{figure}
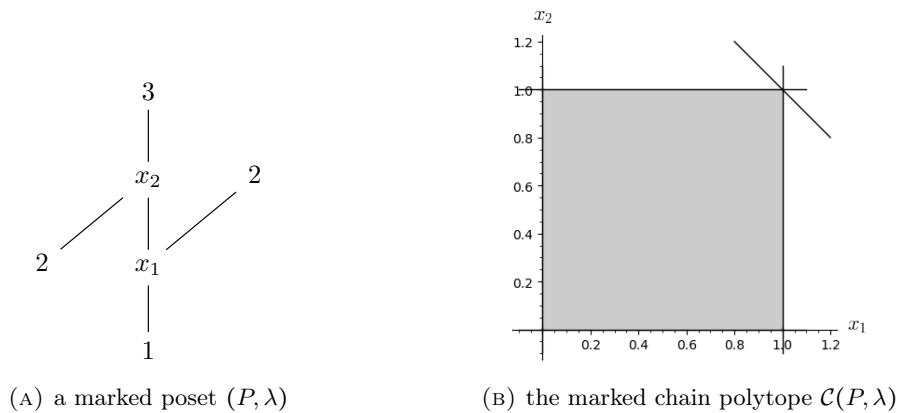

\begin{figure}[h!]
%\begin{tikzcd}[ampersand replacement=\&, column sep = 0.1]
% https://q.uiver.app/#q=WzAsNixbMSwxLCJ4XzIiXSxbMSwyLCJ4XzEiXSxbMiwxLCIyIl0sWzAsMiwiMiJdLFsxLDMsIjEiXSxbMSwwLCIzIl0sWzUsMCwiIiwwLHsic3R5bGUiOnsiaGVhZCI6eyJuYW1lIjoibm9uZSJ9fX1dLFswLDMsIiIsMCx7InN0eWxlIjp7ImhlYWQiOnsibmFtZSI6Im5vbmUifX19XSxbMCwxLCIiLDAseyJzdHlsZSI6eyJoZWFkIjp7Im5hbWUiOiJub25lIn19fV0sWzEsMiwiIiwwLHsic3R5bGUiOnsiaGVhZCI6eyJuYW1lIjoibm9uZSJ9fX1dLFsxLDQsIiIsMCx7InN0eWxlIjp7ImhlYWQiOnsibmFtZSI6Im5vbmUifX19XV0=

  \savebox{\imagebox}{\includegraphics[width=0.32\linewidth]{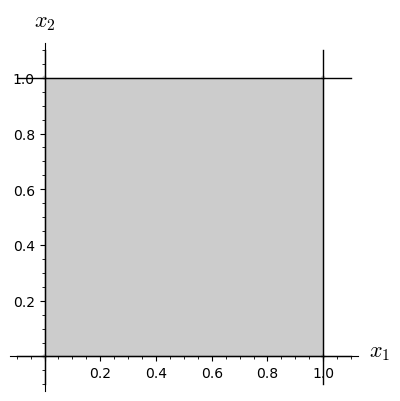}}%  
  \begin{subfigure}[t]{.49\linewidth}
    \centering\raisebox{\dimexpr.5\ht\imagebox-.5\height}{% Raise smaller image into place
      \begin{tikzcd}[ampersand replacement=\&]
	{x_2} \&\& {x_1}
\end{tikzcd}}

\caption{a poset $P'$}
  \end{subfigure}
  \begin{subfigure}[t]{.49\linewidth}
    \centering\usebox{\imagebox}% Place largest image
    \caption{the chain polytope $\CP_{P'}$}
  \end{subfigure}\qquad

  \caption{The poset $P'$ from Example \ref{bsplem3-1} and the corresponding chain polytope $\CP_{P'}$.}
\label{fig:bspchainpos}
\end{figure}

To understand this lemma better, we will look at a small example, to understand the consequences better.

\begin{bsp}\label{bsplem3-1}
We consider the marked chain polytope $\CP(P,\lambda)$ given by the marked poset $(P,\lambda)$ in Figure \ref{fig:bspmchainpos}. The unmarked elements of $P$ are the elements $x_1$ and $x_2$, while the numbers are the markings from $\lambda$ for the corresponding elements in $P^*$.\\
The marked chain polytope $\CP(P,\lambda)$ is constructed by the inequalities
\begin{align*}
&x_i \geq 0~~\text{for } i \in \{1,2\} &\text{ and}\\
&x_i \leq 1~~\text{for } i \in \{1,2\} &\text{ and}\\
&x_1+ x_2 \leq 2.&
\end{align*}
But we see in Figure \ref{fig:bspmchainpos}, that the last inequality does not define the polytope $\CP(P,\lambda)$. Therefore the conditions of Lemma \ref{chainiso} are met.\\
In Figure \ref{fig:bspchainpos} we see, the poset with corresponding chain polytope, which is the same polytope as $\CP(P,\lambda)$.
\end{bsp}

As we said before we do not know much about the faces of marked chain polytopes, but we can prove that there are always some special facets and that $0$ is always a vertex of a marked chain polytope. This will become very useful for the characterisation of marked chain polytopes, but also later for the proof of Theorem \ref{chainorder2l}.

\begin{lemm}\label{minchainfacet}
Let $(P,\lambda)$ be a regular marked poset. For every element $p \in P\backslash P^*$ there exists a facet of $\CP(P,\lambda)$, which is described by $x_p=0$. Therefore $0$ is a vertex of $\CP(P,\lambda)$.
\end{lemm}
\begin{proof}
Let $(P,\lambda)$ be a regular marked poset and $\CP(P,\lambda)$ the respective marked chain polytope. We know from the Definition \ref{defmarkchain}, that $\CP(P,\lambda)$ is described by inequalities of the form
\begin{align*}
x_p \geq 0  \text{ for } p\in P\backslash P^* \text{ and }\sum_{i\in I} x_i \leq c \text{ for } I \subseteq P\backslash P^*~c\in\R_{\geq 0}.
\end{align*}
Since the inequalities $x_p \geq 0$ are independent and no other lower bound can be formed from the inequalities $\sum_{i\in I} x_i \leq c$, the inequality $x_p \geq 0$ is a facet-describing inequality for each $p \in P\backslash P^*$.\\
The intersection of all these facets is the vertex $0$.
\end{proof}

We can now prove a similar characterisation for marked chain polytopes as the characterisation for marked order polytopes.

\begin{theorem}\label{markchain2l}
	Let $(P,\lambda)$ be a regular marked poset.
	The following conditions are equivalent:
	\begin{enumerate}[label=(\alph*)]
	\item $\CP(P,\lambda)$ is a 2-level polytope.
	\item $\CP(P,\lambda)$ is affinely isomorphic to a marked chain polytope, whose markings of each facet describing chain in $P$ only differ by 1. 
	\item $\CP(P,\lambda)$ is affinely isomorphic to a chain polytope.
	\end{enumerate}
	\end{theorem}
	\begin{proof}
	Let $(P,\lambda)$ be a regular marked poset and $d=|P\backslash P^*|$.\\
	Because of Lemma \ref{chainiso} we see \textit{(b)} $\Rightarrow$\textit{(c)}. Since every compressed polytope is a 2-level polytope and chain polytopes are compressed, \textit{(c)} $\Rightarrow$\textit{(a)} is trivial \cite{ohsugi2001compressed}.\\
	Let $\CP(P,\lambda)$ be a 2-level polytope. We want to show \textit{(a)} $\Rightarrow$\textit{(b)}. Therefore we look at the different types of possible describing inequalities again.\\
	From Lemma \ref{minchainfacet} we already know, that $x_p \geq 0$ is a facet for all $p \in P\backslash P^*$. Since $\CP(P,\lambda)$ is a 2-level polytope, all vertices need to fulfill the equation $x_p=0$ or the equation $x_p=c_p$ for $c_p\in\R_{> 0}$ for all $p \in P\backslash P^*$. By scaling every coordinate with $\tfrac{1}{c_p}$ respectively, we get a $(0,1)$-polytope $\mathcal{P}$.\\
	  Assume $\sum_{i\in I}x_i + x_p \leq c$ is a defining inequality of $\mathcal{P}$ with $I\cup \{p\}\subseteq P\backslash P^*$ a maximal chain and $c \geq 2$. We want to show that there exists a vertex $w$ with $w_p=1$ and $w_i=0$ for $\forall i \in P\backslash P^*$. Because we have a $(0,1)$-polytope, $w$ cannot be in the interior of any face of $\mathcal{P}$. Therefore $w$ is a vertex or $w \not\in \mathcal{P}$.\\
	  Assume that $w \not\in \mathcal{P}$. By Separation Theorem \cite[Prop 1.10]{ziegler2006}, there exists a facet defining hyperplane which separates $w$ from the polytope. This is a contradiction, because there only exists defining inequalities like $x_q\geq 0$ for $q\in P\backslash P^*$ and $\sum_{i\in J}x_i \leq c$, where $c \in \N$ and $J$ a chain of $P$. These inequalities cannot separate $w$ from the polytope.\\
	  We now know, that $\mathcal{P}$ has the vertices $0$ and $w$. Because $\mathcal{P}$ is a 2-level polytope and the inequality $\sum_{i\in I}x_i + x_p \leq c$ is a defining inequality, all vertices have to fulfill one of the following equations:
	  \begin{align*}
	  (1)& \sum_{i\in I}x_i + x_p = c &\text{ or }\\
	  (2)& \sum_{i\in I}x_i + x_p = c-d ~~~d\in \N.&
	  \end{align*}
	Since $0$ is a vertex of $\mathcal{P}$, and it does not fulfill the first equation it must fulfill the second equation. The same is true for the vertex $w$, since we assume $c\geq 2$. We can conclude
	\begin{align*}
	\sum_{i\in I}0 + 0 = 0 = c-d = \sum_{i\in I}0 + 1 = 1,
	\end{align*}
	which is absurd. Therefore $c$ needs to be equal to $1$.\\
	In conclusion, all defining facets of $\mathcal{P}$ are of the form
	\begin{align*}
x_i \geq 0 ~~&\text{for } i\in P\backslash P^* &\text{ and }\\ 
\sum_{i\in I} x_i \leq 1 ~~&\text{for } I \subseteq P\backslash P^*.&
\end{align*}
Therefore the markings of each facet describing chain of $P$ only differ by 1, which concludes the proof. 
\end{proof}

\section{2-level Marked Chain-Order Polytope}\label{section5}

For marked chain-order polytopes we also do not know much about their faces, but we can characterise them using marked order and marked chain polytopes.

\begin{lemm}\label{pointchainorder}\textnormal{\cite[Lem 2.7]{fang2018minkowski}}
Let $P$ be a poset with a decomposition $P = P^* \cup C \cup O$ into marked,
chain and order elements. For $x = (\lambda, x_C , x_O) \in \R^P$ we have $x \in \OP_{C,O}(P, \lambda)$ if and only if $x_O \in \OP(P \backslash C, \lambda)$ and $x_C \in \CP(P, \lambda\sqcup x_O)$, where $\lambda \sqcup x_O$ is the map $P^* \cup O \rightarrow \R$
that restricts to $\lambda$ on $P^*$ and to $x_O$ on $O$.
\end{lemm}

With this lemma we can characterise the 2-levelness of marked chain-order polytopes, by using a lot of the results for marked order and marked chain polytopes, we proved before.

\begin{theorem}\label{chainorder2l}
Let $(P,\lambda)$ be a regular marked poset and $P = P^* \cup C \cup O$ a decomposition into marked, chain and order elements. The marked chain-order polytope $\OP_{C,O}(P, \lambda)$ is a 2-level polytope if and only if
	\begin{enumerate}[label=(\alph*)]
	\item $\OP(P\backslash C,\lambda)$ is a 2-level polytope, and
	\item $\OP_{C,O}(P, \lambda)$ is affinely isomorphic to a marked chain-order polytope, whose higher and lower bounds of chain elements only differ by 1.
	\end{enumerate}
\end{theorem}
\begin{proof}
Let us first assume $\OP(P\backslash C,\lambda)$ is a 2-level polytope and $\OP_{C,O}(P, \lambda)$ is affinely isomorphic to a marked chain-order polytope, where the higher and lower bounds of chain elements only differ by 1. We want to show, that $\OP_{C,O}(P, \lambda)$ is a 2-level polytope.\\
Since 2-levelness is preserved under affine transformations, we will use from now on the marked chain-order polytope, where the higher and lower bounds of chain elements only differ by 1 and we will call it $\OP_{C,O}(P, \lambda)$ for simplicity.\\
We will now consider all different cases of facets of $\OP_{C,O}(P, \lambda)$ and will show, that all vertices are in the facet or in a parallel hyperplane.\\
Case 1: Assume $x_p=\lambda(a)$ for $p\in O$ and $a\in P^*$ defines a facet.\\
Since $\OP(P\backslash C, \lambda)$ is a 2-level polytope, all $x_p$ of vertices with $p\in O$ can only have the same upper and lower bound if they are connected in the Hasse-diagram. We call them $\lambda(a)$ and $\lambda(b)$. If these $p$ are only connected to other order elements, they are independent from all chain elements and 2-levelness follows for these elements from the 2-levelness of $\OP(P\backslash C, \lambda)$.\\
If they have a connection to chain elements, $\lambda(a)$ and $\lambda(b)$ only differ by 1. Since $\OP_{C,O}(P, \lambda)$ is a lattice polytope, $x_p$ has only two different possible values. This concludes 2-levelness for facets of the form $x_p=\lambda(a)$.\\
Case 2: Assume $\sum_{i\in I} x_i \leq 1$ is a facet for a chain $I \subseteq C$.\\
From this equation we get $0\leq x_i \leq 1$ for all $i\in I$. It follows directly from this, that all vertices lives in the facet $\sum_{i\in I} x_i = 1$ or $\sum_{i\in I} x_i = 0$. Therefore 2-levelness is true.\\
Case 3: Assume $\sum_{i\in I} x_i + x_p \leq \lambda(b)$ defines a facet for $I\subseteq C$, $p\in O$, $b \in P^*$ and $I$ is a maximal chain between $p$ and $b$.\\
With the given properties, we can follow that all $x_i$ for $i\in I$ and $x_p$ lie between some marked elements $a$ and $b$, such that $\lambda(a)=\lambda(b)-1$. It follows $0\leq x_i \leq 1$ for all $i \in I$ and $\lambda(b)-1\leq x_p \leq \lambda(b)$. In conclusion, all vertices fulfill either $\sum_{i\in I} x_i + x_p = \lambda(b)$ or $\sum_{i\in I} x_i + x_p = \lambda(b)-1$.\\
Case 4: Assume $\sum_{i\in I} x_i + \lambda(a) \leq x_p$ defines a facet for $I\subseteq C$, $p\in O$, $a \in P^*$ and $I$ is a maximal chain between $a$ and $p$.\\
This case is analogous to Case 3.\\
Case 5: Assume $\sum_{i\in I} x_i + x_p \leq x_q$ defines a facet for $I\subseteq C$, $p,q\in O$ and $I$ is a maximal chain between $p$ and $q$.\\
Because $\OP(P\backslash C,\lambda)$ is a 2-level polytope, $x_p$ and $x_q$ lie between the same upper and lower bounds $\lambda(a)$ and $\lambda(b)$. Since the chain elements of the chain also lie between these two bounds, we know $\lambda(a)=\lambda(b)-1$. Therefore $x_p$ and $x_q$ can differ at most by at most $1$ and in addition $0\leq x_i \leq 1$ for all $i\in I$.\\
In conclusion all vertices fulfill either $\sum_{i\in I} x_i + x_p = x_q$ or $\sum_{i\in I} x_i + x_p +1 = x_q$.\\
Case 6: Assume $x_c\geq 0$ for $c \in C$ is a facet.\\
The coordinate $x_c$ must be bounded by one of the inequalities from one of the cases before, because otherwise $\OP_{C,O}(P,\lambda)$ would not be a polytope. In each of these cases $x_c$ lies between $0$ and $1$. It follows trivial, that all vertices fulfill either $x_c=0$ or $x_c=1$.\\
In conclusion, we can see for all facets all vertices lie in the facet or a parallel hyperplane. Therefore, $\OP_{C,O}(P,\lambda)$ is a 2-level polytope. This concludes the first implication of the proof.\\

Assume that $\OP_{C,O}(P,\lambda)$ is a full-dimensional 2-level polytope (coordinates which are fixed can be ignored).\\
We know from Lemma \ref{minchainfacet} that $0$ is a vertex of all marked chain polytopes. With Lemma \ref{pointchainorder} we can see, that for all $x \in \OP(P\backslash C,\lambda)$ the point $(0,x)$ is a point of the marked chain-order polytope. Also for the same reason as in the proof of Lemma \ref{minchainfacet} we see, that $x_c\geq 0$ defines a facet of $\OP_{C,O}(P,\lambda)$ for all $c \in C$. Combining these two properties we get
\begin{align*}
\OP_{C,O}(P,\lambda) \cap \bigcap_{c \in C} H_{x_c=0}=\{0\}\times \OP(P\backslash C,\lambda)
\end{align*}
and $\OP(P\backslash C, \lambda)$ is a face of $\OP_{C,O}(P,\lambda)$. Since every face of a 2-level polytope is a 2-level polytope, $\OP(P\backslash C, \lambda)$ needs to be a 2-level polytope.\\
Because $x_c\geq 0$ defines a facet for all $c\in C$, again for all $c \in C$ all vertices satisfies the equation $x_c=0$ or $x_c=d_c$ for some $d_c \in \R$. We can now scale all chain element coordinates by $\tfrac{1}{d_c}$. We will use the same name from here on for the resulting marked chain-order polytope.\\
Consider a facet $\sum_{i\in I}x_i +x_p \leq \lambda(b)$ for a chain $I\subseteq C$ with starts in $p\in O$ and ends in $b\in P^*$. Let $\lambda(a)$ be the lower bound for $x_p$ and the elements in the chain. Since $\OP(P\backslash C,\lambda)$ is a face of $\OP_{C,O}(P,\lambda)$, there exists a vertex with $x_p=\lambda(a)$ and $x_i=0$ for all $i\in I$. Therefore all vertices have to fulfill either
\begin{align*}
(1)& \sum_{i\in I}x_i + x_p = \lambda(b) &or\\
(2)& \sum_{i\in I}x_i + x_p = \lambda(a).&
\end{align*}
Also because $x_p\geq \lambda(a)$ and $x_c\leq 1$ for one $c\in I$, the set
\begin{align*}
\OP_{C,O}(P,\lambda)\cap H_{x_p=\lambda(a)} \cap H_{x_c=1}\bigcap_{i \in I\backslash\{c\}} H_{x_i=0}
\end{align*}
is a face of $\OP_{C,O}(P,\lambda)$. If this face is not empty, there exists a vertex $w$ with $w_p=\lambda(a)$, $w_c=1$ and $w_i=0$ for $i \in I\backslash\{c\}$.\\
To show that this face is not empty, assume there exists no point $z \in \OP(P\backslash C,\lambda)$ with $z_p=\lambda(a)$, such that $\exists c \in I$ with $v_c=1$ and $v_i=0$ for $i \in I\backslash\{c\}$ and $v \in \CP(P, \lambda\sqcup z)$. That means, that for every $v_c$, there exists a chain with $c$ inside and the same upper and lower bound. Because of the full-dimensionality, this can only happen if at least one of these bounds is not $z_p=\lambda(a)$ and a bound introduced by a coordinate of $z$. But all coordinates of $z$ except $z_p$ can choose the two values $\lambda(a)$ and $\lambda(b)$. Therefore the assumption is not possible, because we can choose the coordinates of $z$, such that there is a chain with different upper and lower bounds.\\
We conclude, there exists a vertex $w$ with $w_p=\lambda(a)$, $w_c=1$ and $w_i=0$ for $i \in I\backslash\{c\}$. This vertex cannot lie in the hyperplane described by $\sum_{i\in I}x_i + x_p = \lambda(a)$. It holds
\begin{align*}
\sum_{i\in I}w_i + w_p = 1 + \lambda(a) =\lambda(b).
\end{align*}
For facets of the form $\sum_{i\in I}x_i +\lambda(a) \leq x_p$ for a chain $I\subseteq C$ which starts in $a\in P^*$ and ends in $p\in O$, the proof that the upper bound and lower bound only differ by 1 is analogous.\\
The last facet we need to consider is a facet of the form $\sum_{i\in I}x_i +x_p \leq x_q$ for a chain $I\subseteq C$ with starts in $p\in O$ and ends in $q\in O$. The polytope $\OP(P\backslash C,\lambda)$ is 2-level, therefore $x_p$ and $x_q$ have the same upper and lower bound. With this we can do a similar proof as before, to get that the bounds can only differ by 1.
This concludes the proof.
\end{proof}

This theorem also includes Theorems \ref{markorder2l} and Theorem \ref{markchain2l}, if we choose $C=\emptyset$ or $O=\emptyset$ respectively.

\section{Ehrhart polynomial of marked order polytopes}\label{section6}

For an integral polytope $Q \subseteq \R^d$, whose vertices have integer coordinates. We denote the number of lattice points in $Q$ dilated by a factor $n \in \Z_{\geq 0}$ as $$\text{Ehr}_{Q}(n):= |nQ \cap \Z^d|.$$
Ehrhart \cite{ehrhart1962polyedres} discovered that $\text{Ehr}_{Q}(n)$ is a polynomial in $n$ of degree $\dim(Q)$. Therefore, we call $\text{Ehr}_{Q}(n)$ the \textit{Ehrhart polynomial} of $Q$.\\

While working with marked order polytopes to prove the theorems of previous sections, we found a concrete formula for the Ehrhart polynomial of the marked order polytope. Since it is similar to the formula of order polytopes which was proven by Stanley \cite[Thm 4.5.14]{stanley1997enumerative}, we will use some of his techniques and results.

\begin{defi}\textnormal{\cite[Sec 3.5 and 3.12]{stanley1997enumerative}}
Let $P$ be a finite poset with $n$ elements. We call the bijective order preserving maps $\pi:P\rightarrow [n]$ \textit{linear extensions} of $P$.\\
We identify each linear extension $\pi$ with a tuple $(a_1,\ldots,a_n)\in P^n$, which is defined by $$\pi \rightarrow (\pi^{-1}(1),\ldots,\pi^{-1}(n)).$$ We denote the set of all these tuple for a poset $P$ with $\mathcal{L}(P)$.\\
A \textit{descent} of $\pi$ is an index $j \in [n]$ for which $a_{j+1} \prec a_j$ holds. We denote the number of descents of $\pi$ by $d(\pi)$.
\end{defi}

We can think of linear extensions as sorting the elements of a poset linearly in such a way that the order of the poset is preserved in this sorting.

\begin{lemmdefi}\textnormal{\cite[Lem 4.5.1]{stanley1997enumerative}}\label{uniqext}
Let $P$ be a finite poset with $n$ elements, $f:P\rightarrow \Z$ an order-preserving map and $\pi\in \mathcal{L}(P)$. We say that $f$ is \textit{compatible} with $\pi$ if
	\begin{align*}
	(1)&~f(a_1)\leq f(a_2)\leq \ldots \leq f(a_n) \text{ and}\\
	(2)&~f(a_i) < f(a_{i+1})\text{ for } a_i \prec a_{i+1} \text{ in } P
\end{align*}
holds. For every order-preserving map $f$, there exists a unique compatible linear extension $\pi$.
\end{lemmdefi}

\begin{lemm}\textnormal{\cite[Proof of 4.5.14]{stanley1997enumerative}}\label{compaquiv}
	Let $P$ be a finite poset with $n$ elements. An order-preserving map $f:P\rightarrow [m]$ is compatible with $\pi=(a_1,\ldots,a_n)\in \mathcal{L}(P)$ if and only if  
	\begin{align*}
		m-d(\pi)\geq f(a_n)-d_n\geq \cdots\geq f(a_1)-d_1\geq 1
	\end{align*}
	holds, where $d_i$ is defined for all $i\in [n]$ as  
	\begin{align*}
		d_i:=|\{j\in \N~:~j\leq i, a_{j+1} \prec a_j\}|.
	\end{align*}
\end{lemm}

\begin{propdefi}
	Let $(P,\lambda)$ be a regular marked poset and\\
	$\pi\in \mathcal{L}(P)$. We define an order $\preceq_\pi$ on $P$ by  
	$$
	p\preceq_\pi q \text{ if and only if } \pi(p)\leq \pi(q).
	$$
	This defines a linear order on the underlying set of $P$. We denoted this new poset by $P_\pi$.
\end{propdefi}

\begin{defi}
	Let $(P,\lambda)$ be a regular marked poset and $\pi\in \mathcal{L}(P)$. A \textit{marked chain} $I$ in $P_\pi$ is a maximal chain $I$ in $P_\pi$ consisting of elements from $P\backslash P^*$ that lies between two elements $a,b\in P^*$. These elements $a$ and $b$ are unique since $P_\pi$ orders the elements linearly. We denote the marked elements which bounds a marked chain $I$ by $a_I$ and $b_I$.
\end{defi}

With these definitions, we can give an explicit description of the Ehrhart polynomial.

\begin{theorem}\label{ehrmarkorder}
	Let $(P,\lambda)$ be a regular marked poset and $\mathcal{L}(P)$ the set of all linear extensions of $P$ where the markings are ordered in increasing order.\\
	The Ehrhart polynomial of $\OP(P,\lambda)$ is
\begin{align*}
	\text{Ehr}_{\OP(P,\lambda)}(m)=\sum_{\pi \in \mathcal{L}(P)}\prod_{\substack{I\subseteq P_\pi\\
			\text{marked chain}}} \binom{m(\lambda(b_I)-\lambda(a_I))-d+|I|}{|I|}.
\end{align*}
Here, $d$ is the number of descents of $\pi$ between $a_I$ and $b_I$.
\end{theorem}
\begin{proof}
Let $(P,\lambda)$ be a regular marked poset.
At first, it is trivial to see, that the number of lattice points of $m\OP(P,\lambda)$ is equal to the number of order preserving maps $f:P\rightarrow \Z$ with $f(a)=m\cdot\lambda(a)$ for all marked elements $a\in P^*$. Therefore we will count these maps.\\
We add the cover $a \prec b$ for all $a,b\in P^*$ if $\lambda(a)<\lambda(b)$ to the poset $P$. We call the new poset also $P$ for easier notations.\\
As a result, all linear extension of $P$ keep the markings of $P^*$ in increasing order.\\
Let $P^*=\{a_1,a_2,\ldots,a_z\}$ be the marked elements, ordered increasingly by their marking. We know from Lemma \ref{compaquiv}, that a order preserving map $f:P\rightarrow [m\cdot\lambda(a_z)]$ with $f(a_i)=m\cdot\lambda(a_i)$ $\forall i \in[z]$ is compatible to a linear extension $\pi$, if and only if
\begin{align*}
m\cdot\lambda(a_z)-d(\pi)\geq f(x_n)-d_n\geq \cdots\geq f(x_1)-d_1\geq 1
\end{align*}
holds. If we insert the already fixed values $f(a_i)=m\cdot\lambda(a_i)$, we can see, that for each marked chain $I$ between $a_i$ and $a_{i+1}$, the map $f$ has to satisfy
%\begin{align*}
%m\cdot\lambda(a_z)-d(\pi)=m\cdot\lambda(a_z)-d_n\geq f(x_{n-1})-d_{n-1}\geq \cdots f(x_l)-d_l\geq  m\cdot\lambda(a_{z-1})-d_{l-1}\geq \cdots \geq m\cdot\lambda(a_1)-d_1 \geq 1.
%\end{align*}
\begin{align*}
	m\cdot\lambda(a_{i+1})-d_t\geq f(x_{t-1})-d_{t-1}\geq \cdots \geq f(x_{t-k})-d_{t-k}\geq  m\cdot\lambda(a_{i})-d_{t-(k+1)}.
\end{align*}
Hence, there are
\begin{align*}
	&\left(\binom{m(\lambda(a_{i+1})-\lambda(a_i))-(d_t-d_{t-(k+1)})-1}{k}\right)\\=&\binom{m(\lambda(a_{i+1})-\lambda(a_i))-(d_t-d_{t-(k+1)})+k}{k}
\end{align*}
many possibilities for choosing the values $f(x_j)$ with $j\in\{t-1,\ldots,t-k\}$. For the marked chain $I$ in $P_\pi$ with $k$ elements $d:=(d_t-d_{t-(k+1)})$ is the number of descents starting in $a_i$ and ending in $a_{i+1}$. Therefore the number of possibly maps $f$ which are comparable to $\pi$ is
\begin{align*}
	\prod_{\substack{I\subseteq P_\pi\\ \text{marked chain with}\\
			d \text{ descents}}} \binom{m(\lambda(b_I)-\lambda(a_I))-d+|I|}{|I|}.
\end{align*}
Since each order preserving map $f:P\rightarrow \Z$ with $f(a)=m\cdot\lambda(a)$ for all marked elements $a\in P^*$ is comparable to one unique linear extension by Lemma \ref{uniqext}, we get for the total number of such maps
\begin{align*}
	\sum_{\pi \in \mathcal{L}(P)}\prod_{\substack{I\subseteq P_\pi\\ \text{marked chain with}\\
			d \text{ descents}}} \binom{m(\lambda(b_I)-\lambda(a_I))-d+|I|}{|I|}.
\end{align*}
\end{proof}

With this form of the Ehrhart polynomial we see, that the complicated part to compute from this formula is the product. This product does not exists if we only have one chain in each of the linear extensions. Hence, we only have one maximal and one minimal marked element, which by Theorem \ref{markorder2l} only happens for the simple case of marked order 2-level polytopes.

\begin{figure}[t!]
    \begin{subfigure}[b]{0.49\linewidth}
        \centering
        $\vcenter{\hbox{
        \begin{tikzcd}[ampersand replacement=\&]
            {a_6} \\
            \&\& {x_6} \\
            {a_5} \\
            \&\& {x_5} \\
            {a_4} \\
            \&\& {x_4} \&\& {x_2} \\
            {a_3} \\
            \&\& {x_3} \\
            {a_2} \\
            \&\& {x_1} \\
            {a_1}
            \arrow[no head, from=11-1, to=10-3]
            \arrow[no head, from=9-1, to=10-3]
            \arrow[no head, from=8-3, to=7-1]
            \arrow[no head, from=9-1, to=8-3]
            \arrow[no head, from=1-1, to=2-3]
            \arrow[no head, from=3-1, to=2-3]
            \arrow[no head, from=3-1, to=4-3]
            \arrow[no head, from=5-1, to=4-3]
            \arrow[no head, from=5-1, to=6-3]
            \arrow[no head, from=7-1, to=6-3]
            \arrow[no head, from=10-3, to=6-5]
            \arrow[no head, from=2-3, to=6-5]
        \end{tikzcd}
        }}$
        \caption{the marked poset $P_6$}
\label{fig:bspehrpos1}
    \end{subfigure}
    \begin{subfigure}[b]{0.49\linewidth}
        \centering
        $\vcenter{\hbox{
        \begin{tikzcd}[ampersand replacement=\&, baseline=-16\the\dimexpr\fontdimen22\textfont2\relax]
            {a_3} \\
            \&\& {x_3} \\
            {a_2} \&\&\&\& {x_2} \\
            \&\& {x_1} \\
            {a_1}
            \arrow[no head, from=5-1, to=4-3]
            \arrow[no head, from=3-1, to=4-3]
            \arrow[no head, from=4-3, to=3-5]
            \arrow[no head, from=3-5, to=2-3]
            \arrow[no head, from=2-3, to=1-1]
            \arrow[no head, from=3-1, to=2-3]
        \end{tikzcd}
        }}$
        \caption{the marked poset $P_3$}
\label{fig:bspehrpos2}
    \end{subfigure}
\caption{Examples of the marked posets from Example \ref{bsplem5-1}.}
\label{fig:bspehrpos}
\end{figure}
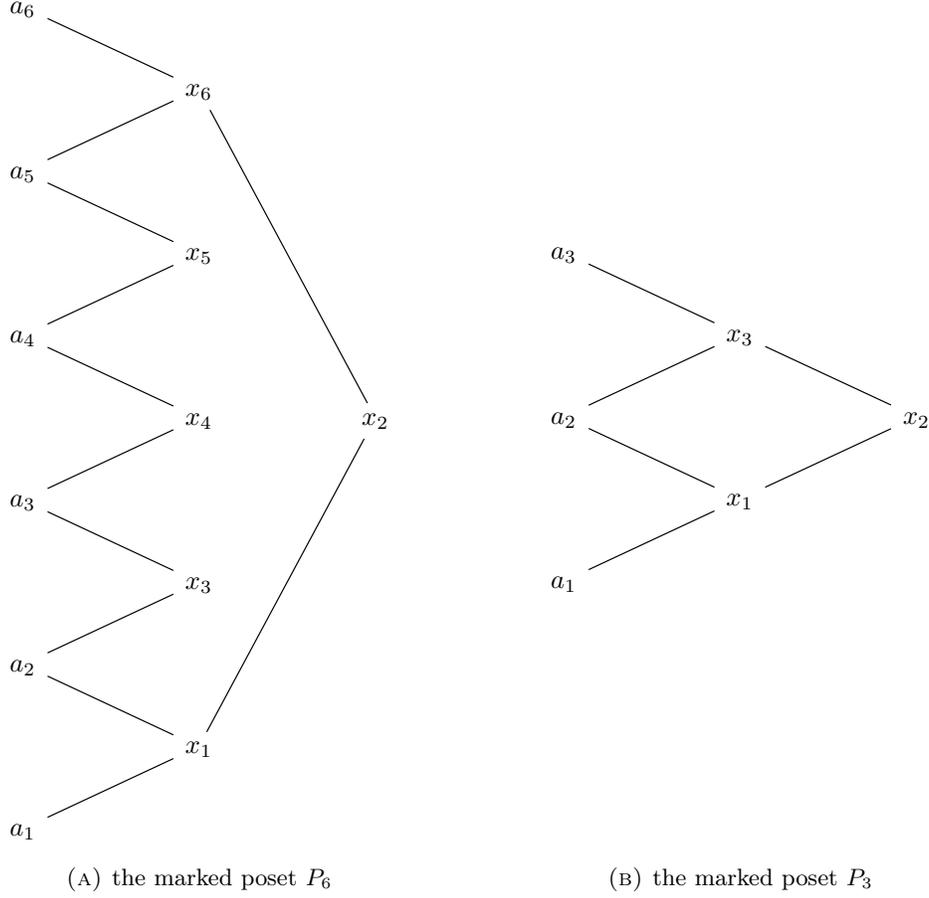

Another way not to have a problem with the product, is when there are a lot of similar linear extensions and the product becomes an exponent. We construct a class of marked posets, which have these exact properties. %We can even calculate the $h^*$ polynomial for these marked order polytopes.

\begin{bsp}\label{bsplem5-1}
We consider the marked poset $$P_m=\{a_1,\ldots,a_m\}\cup\{x_1,\ldots,x_m\}$$ with $a_i\prec x_{i+1}$ and $x_i\prec a_{i}$ for $i \in \{2,\ldots ,m-1\}$, $a_1\prec x_1 \prec a_2$ and $x_1\prec x_2 \prec x_m$. We define a marking $\lambda: \{a_1,\ldots,a_m\}\rightarrow \Z$ on $P_m$ with $\lambda(a_{i+1})-\lambda(a_i)=c$ for some $c\in \R$ and all $i\in [m]$.\\
In Figure \ref{fig:bspehrpos} we can see examples of these posets for $m=6$ and $m=3$.\\
To calculate the Ehrhart polynomial for these marked posets, we need to understand the linear extensions of these posets and the chains the linear extensions are building between the marked elements.\\
At first, we define $I_i$ as the chain of unmarked elements between the marked elements $a_i$ and $a_{i+1}$ in a linear extension. We can observe, that always $x_1 \in I_1$ and $x_{i+1} \in I_{i}$ for $i\in [m-2]$. Therefore, $x_2$ is the only element which is not fixed.\\
In addition, we observe, that only if $x_2 \in I_1$, the linear extension does not have a descent in the chain which contains $x_2$. Hence, in each linear extension we have $(m-2)$ chains with only $1$ unmarked element and no descent in that chain, and one chain with $2$ unmarked elements, with 0 descents if $x_2 \in I_1$ in the linear extension or with 1 descent if $x_2 \in I_{i+1}$ for $i\in [m-2]$ in the linear extension.\\
First look at the cases, where $x_2\in I_{i+1}$ for $i\in [m-2]$. There are $2\cdot(m-3)$ cases, where $x_2 \in I_{i}$ can be above or below the elements $x_{i+1}$ for $i\in \{2,\ldots,m-2\}$ in the chain. And there is one case where $x_2 \in I_{m-1}$, because $x_2 \prec x_m$. Hence, we have $2(m-3)+1=2m-5$ linear extensions, where $x_2\in I_{i+1}$ for $i\in [m-2]$.\\
For each of these linear extension we calculate the product of Theorem \ref{ehrmarkorder} as follows,
\begin{align*}
\left[\prod_{i=1}^{m-2}\binom{n\cdot c-0+1}{1}\right]\cdot \binom{n\cdot c-1+2}{2}=
(nc+1)^{m-2}\frac{(nc+1)nc}{2}=
(nc+1)^{m-1}\frac{nc}{2}.
\end{align*}
There is one linear extension, where $x_2\in I_1$. The product for this case is
\begin{align*}
\left[\prod_{i=1}^{m-2}\binom{n\cdot c-0+1}{1}\right]\cdot \binom{n\cdot c-0+2}{2}=
(nc+1)^{m-2}\frac{(nc+2)(nc+1)}{2}=
(nc+1)^{m-1}\frac{nc+2}{2}.
\end{align*}
We can put this together to get the complete Ehrhart polynomial of this poset,
\begin{align*}
	\text{Ehr}_{\OP(P,\lambda)}(n)=&\sum_{\pi \in \mathcal{L}(P)}\prod_{\substack{I\subseteq P_\pi\\
			\text{marked chain}}} \binom{n(\lambda(b_I)-\lambda(a_I))-d+|I|}{|I|}\\
	=&(2m-5)\cdot((nc+1)^{m-1}\frac{nc}{2})+((nc+1)^{m-1}\frac{nc+2}{2})\\
	=&(m-2)nc(nc+1)^{m-1}+(nc+1)^{m-1}.
\end{align*}
%With this form, we can also calculate the $h^*$ polynomial. For this we denote $S(n,k)$ as the stirling numbers of second order. By using the binomial Theorem on the generating function for $\sum_{n\geq 0}n^mt^n$, we get the $h^*$ polynomial
%\begin{align*}
%h^*(t)=\sum_{k=0}^{m-1}\binom{m-1}{k}c^k\sum_{j=0}^kS(k+1,j+1)(-1)^{k-j}j!(1-x)^{m-j}&+\\
%(m-2)\sum_{k=1}^{m}\binom{m-1}{k-1}c^k\sum_{j=0}^kS(k+1,j+1)(-1)^{k-j}j!(1-x)^{m-j}&.
%\end{align*}
For $m=3$ we have the poset in Figure \ref{fig:bspehrpos}. In this case the the Ehrhart polynomial looks even nicer. We get the Ehrhart polynomial
$$
\text{Ehr}_{\OP(P_3,\lambda)}(n)=(nc+1)^3.
$$
Neither the marked order polytope nor the marked chain polytope, is the dilated cube. We can observe this, by comparing the number of vertices.
\end{bsp}

\end{document}